\documentclass[final,1p,times]{elsarticle}

\usepackage{amsmath,amsthm,amssymb}

\newtheorem{prop}{Proposition}[section]
\newtheorem{thm}[prop]{Theorem}
\newtheorem{corol}[prop]{Corollary}
\newtheorem{conj}[prop]{Conjecture}
\newtheorem{lem}[prop]{Lemma}

\theoremstyle{definition}

\newtheorem{ex}[prop]{Example}

\theoremstyle{remark}
\newtheorem{rem}[prop]{Remark}
\newtheorem{conv}[prop]{Convention}

\makeatletter
\def\ps@pprintTitle{%
     \let\@oddhead\@empty
     \let\@evenhead\@empty
     \def\@oddfoot{\footnotesize\itshape Linear Algebra and its Applications 569 (2019) 241--265\hfill}
     \let\@evenfoot\@oddfoot}
\makeatother

\begin{document}

\begin{frontmatter}



\title{A generalization of circulant Hadamard and conference matrices}


\author[UJF,OU,KUT]{Ond\v{r}ej Turek} \ead{ondrej.turek@osu.cz}

\address[UJF]{Nuclear Physics Institute, Czech Academy of Sciences, 250 68 \v{R}e\v{z}, Czech Republic}
\address[OU]{Department of Mathematics, Faculty of Science, University of Ostrava, 30.\ dubna 22, 701 03 Ostrava, Czech Republic}
\address[KUT]{Laboratory for Unified Quantum Devices, Kochi University of Technology, Kochi 782-8502, Japan}

\author[UdA]{Dardo Goyeneche} \ead{dardo.goyeneche@uantof.cl}

\address[UdA]{Departamento de F\'{i}sica, Facultad de Ciencias B\'{a}sicas, Universidad de Antofagasta, Casilla 170, Antofagasta, Chile}

\begin{abstract}
We study the existence and construction of circulant matrices $C$ of order $n\geq2$ with diagonal entries $d\geq0$, off-diagonal entries $\pm1$ and mutually orthogonal rows. These matrices generalize circulant conference ($d=0$) and circulant Hadamard ($d=1$) matrices. We demonstrate that matrices $C$ exist for every order $n$ and for $d$ chosen such that $n=2d+2$, and we find all solutions $C$ with this property. Furthermore, we prove that if $C$ is symmetric, or $n-1$ is prime, or $d$ is not an odd integer, then necessarily $n=2d+2$. Finally, we conjecture that the relation $n=2d+2$ holds for every matrix $C$, which generalizes the circulant Hadamard conjecture. We support the proposed conjecture by computing all the existing solutions up to $n=50$.

\end{abstract}


\end{frontmatter}

\section{Introduction}

A \emph{circulant matrix} is a square matrix in which each row is obtained as a cyclic shift of the precedent row by one position to the right. That is, a circulant matrix of order $n$ takes the form
\begin{equation}\label{C}
C=\left(\begin{array}{ccccc}
c_{0} & c_{1} & \cdots & c_{n-2} & c_{n-1} \\
c_{n-1} & c_{0} & c_{1} &  & c_{n-2} \\
\vdots & c_{n-1} & c_{0} & \ddots & \vdots \\
c_{2} &  & \ddots & \ddots & c_{1} \\
c_{1} & c_{2} & \cdots & c_{n-1} & c_{0}
\end{array}\right)\,.
\end{equation}
A circulant matrix is fully specified by its first row, $(c_0,c_1,\ldots,c_{n-1})$, which we call the \emph{generator} of $C$.

Let us consider two special types of real circulant matrices, namely
\begin{itemize}
\item
\emph{circulant Hadamard matrices}, defined by conditions $c_j\in\{1,-1\}$ for $j=0,1,\ldots,n-1$ and $CC^T=nI$ (the superscript $T$ denotes transposition);
\item
\emph{circulant conference matrices}, defined by conditions $c_j\in\{1,-1\}$ for $j=1,\ldots,n-1$, $c_0=0$ and $CC^T=(n-1)I$.
\end{itemize}
The \emph{circulant Hadamard conjecture} says that circulant Hadamard matrices exist only for $n=1$ and $n=4$. The conjecture is open already for over half a century: according to Schmidt~\cite{Sch1}, ``the conjecture was first mentioned in Ryser's book~\cite{Ry} (1963), but goes back further to obscure sources''. Turyn~\cite{Tu} proved in 1965 that $n$ can only take values $4u^2$ for an odd $u$ and derived further necessary conditions on $n$. Schmidt~\cite{Sch1,Sch2} showed that the circulant Hadamard conjecture is true for orders up to $n=10^{11}$ with three possible exceptions. On top of these results, it is known that a circulant Hadamard matrix cannot be symmetric for $n>4$ (Johnsen~\cite{Jo}, Brualdi and Newman~\cite{Br}, McKay and Wang~\cite{MKW}, Craigen and Kharaghani~\cite{CK}).

By contrast, the problem of existence of circulant conference matrices is fully solved. Stanton and Mullin~\cite{SM} demonstrated that circulant conference matrices only exist of order $n=2$; later Craigen~\cite{Cr} proposed a simpler proof of this fact.

The two kinds of matrices described above serve as a main motivation for our paper. We are concerned with their common generalization, in which we allow the diagonal entries of the matrix $C$ to take an arbitrary value $d\in\mathbb{R}$. For the sake of convenience, we assume $d\geq0$ without loss of generality, and we exclude the trivial case $n=1$. The aim of our work is thus to study matrices $C$ defined by the following conditions:
\begin{equation}\label{Conditions}
\left\{
\begin{array}{l}
\text{$C$ is a circulant matrix of order $n\geq2$ with generator $(c_0,c_1,\ldots,c_{n-1})$}\,; \\
c_j\in\{1,-1\} \quad\text{for all}\; j=1,\ldots,n-1 \,; \\
c_0=d\geq 0\,; \\
CC^T=(d^2+n-1)I\,.
\end{array}
\right.
\end{equation}
Matrices $C$ for $d=1$ and $d=0$ correspond to circulant Hadamard matrices and circulant conference matrices, respectively. In this paper we find all matrices satisfying~\eqref{Conditions} for any value $d\geq0$ that is not an odd integer. The case of $d$ being odd involves the circulant Hadamard conjecture and is thus much harder; for that case we conjecture that all matrices obeying conditions~\eqref{Conditions} satisfy the relation $n=2d+2$. We verify the conjecture up to $n=50$.

There exists another generalization of circulant Hadamard and conference matrices called circulant weighing matrices. A weighing matrix $W$ of order $n$ and weight $k$ is an $n\times n$ matrix having entries from the set $\{0,1,-1\}$ such that $WW^T=kI$. Circulant weighing matrices and their classification were studied by several authors, see works of Eades and Hain~\cite{EH}, Arasu et al.~\cite{AS,ALMNR}, Ang et al.~\cite{AAMS}.

To the best of our knowledge, matrices obeying \eqref{Conditions} for a general $d$ have not been studied before. However, similar parametric matrix problems without the circulancy assumption were already considered. Seberry and Lam~\cite{SL} examined symmetric matrices with orthogonal rows having a constant $m$ on the diagonal and $\pm1$ off the diagonal, and Lam~\cite{La} later extended the study to the non-symmetric case. Recently, Hermitian unitary $n\times n$ matrices with $\pm d/\sqrt{d^2+n-1}$ on the diagonal and complex numbers of modulus $1/\sqrt{d^2+n-1}$ off the diagonal were studied in mathematical physics in relation to scattering in quantum graph vertices (Turek and Cheon \cite{TC}, Kurasov and Ogik \cite{KO}). Also, elements of a special class of real symmetric matrices having constant diagonal $d$, off-diagonal entries $\pm1$ and orthogonal rows are one-to-one related to real equiangular tight frames \cite{GT}.

Our matrices are also closely related to Barker sequences. A Barker sequence is a finite sequence of $n$ numbers $\{c_k\}$ with $c_k\in\{-1,1\}$ and $0\leq k \leq n-1$ which satisfies $|\sum_{k=0}^{n-m-1}c_kc_{k+m}|\leq1$ for every $0\leq m\leq n-1$. It has been proven that only eight Barker sequences exist for length $n\leq13$ \cite{BM08}, if we assume $c_0=c_1=1$ without loss of generality. Furthermore, the existence of a Barker sequence of length $n>13$ would imply that a circulant Hadamard matrix of size $n$ exists (see \cite[Chapter VI, \S14]{BJL99}). This means that Barker sequences of length $n>13$ necessarily imply perfect autocorrelation for the sequence. We say that an autocorrelation is perfect if $\sum_{k=0}^{n-1}c_kc_{k+m \bmod n}=0$ for every $1\leq m\leq n-1$. Sequences with low autocorrelation have a fundamental importance in radar signals theory \cite{CB67}, data transmission and data compression \cite{H52}.
It is thus interesting to search for new finite sequences having perfect autocorrelation, in a similar way as Huffman generalized Barker sequences  \cite{H62}. With this aim, in the present work we define sequences having the first element $c_0\geq0$ different from one in general, i.e., the sequence $\{c_k\}$ does not have all its elements with constant amplitude. This perturbation in the amplitude of the signal allows us to find interesting novel results for sequences of any length $n$. From the point of view of correlations of finite sequences the main result of our paper can be stated as follows: We find the complete set of sequences $\{c_k\}$ of length $n$ with $c_0=n/2-1$, $c_k\in\{-1,1\}$ for $1\leq k\leq n-1$ and having perfect autocorrelation. These sequences exist for every $n\geq2$. Furthermore, we conjecture that every finite sequence of length $n$, with $c_0\geq0$, $c_k\in\{-1,1\}$ for $1\leq k\leq n-1$ and having perfect autocorrelation satisfies $c_0=n/2-1$. If this conjecture is true, then Barker sequences of length $n>13$ do not exist.

The paper is organized as follows. In Section~\ref{Preliminaries} we review basic properties of matrices $C$ satisfying conditions~\eqref{Conditions}. In particular, we prove that a matrix $C$ of order $n$ with diagonal entries $d$ exists only if $n\geq2d+2$. In Section~\ref{Section d<} we derive further necessary conditions and bring in additional results obtained by a computer calculation. On the basis of our findings, we formulate a conjecture that extends the circulant Hadamard conjecture. In Section~\ref{Section: symmetric} we prove that a symmetric matrix $C$ of order $n$ with diagonal entries $d$ exists if and only if $n=2d+2$. Finally, in Section~\ref{Section d=} we find all matrices $C$ that obey conditions~\eqref{Conditions} and have the property $n=2d+2$.

\section{Preliminaries}\label{Preliminaries}

Let $C$ be a circulant matrix of order $n$.
The vectors
$$
v_k=\frac{1}{\sqrt{n}}\left(1,\omega^k,\omega^{2k},\ldots,\omega^{(n-1)k}\right)^T\,,
$$
where $\omega=\mathrm{e}^{2\pi\mathrm{i}/n}$, are normalized eigenvectors of $C$ for all $k=0,1,\ldots,n-1$.
If the matrix $C$ has generator $(c_0,c_1,\ldots,c_{n-1})$, then the corresponding eigenvalues of $C$ are
\begin{equation}\label{eigenvalues}
\lambda_k=c_0+c_{1}\omega^k+c_{2}\omega^{2k}+\cdots+c_{n-1}\omega^{(n-1)k}\,.
\end{equation}
Since the vector $(\lambda_0,\lambda_1,\ldots,\lambda_{n-1})^T$ is obtained as the discrete Fourier transform (DFT) of $(c_0,c_1,\ldots,c_{n-1})^T$, the values $c_\ell$ can be expressed using the inverse DFT as follows:
\begin{equation}\label{invDFT}
c_j=\frac{1}{n}\left(\lambda_0+\lambda_1\omega^{-j}+\lambda_2\omega^{-2j}+\cdots+\lambda_{n-1}\omega^{-(n-1)j}\right)\,.
\end{equation}
From now on we focus on circulant matrices $C$ with generator $(c_0,c_1,\ldots,c_{n-1})$ satisfying 
conditions~\eqref{Conditions}.
For the sake of convenience, we will adopt the following convention.

\begin{conv}\label{Numbering}
The rows and columns of $C$ are indexed from $0$ to $n-1$, i.e., they will be referred to as $0$th, $1$st,\ldots, $(n-1)$th.
\end{conv}

\begin{prop}
If $C$ satisfies conditions~\eqref{Conditions}, then
\begin{equation}\label{sum same}
\left|\sum_{j=0}^{n-1}c_j\right|=\sqrt{d^2+n-1}\,.
\end{equation}
Moreover, if $n$ is even, then
\begin{equation}\label{sum alternating}
\left|\sum_{j=0}^{n-1}(-1)^j c_j\right|=\sqrt{d^2+n-1}\,.
\end{equation}
\end{prop}

\begin{proof}
The assumption $CC^T=(d^2+n-1)I$ implies that the eigenvalues $\lambda_k$ of $C$, given by equation~\eqref{eigenvalues}, obey $|\lambda_k|=\sqrt{d^2+n-1}$ for all $k=0,1,\ldots,n-1$. In the special case $k=0$ we obtain equation~\eqref{sum same}. If $n$ is even, then $k=\frac{n}{2}$ leads to equation~\eqref{sum alternating}.
\end{proof}

\section{Relations between the order $n$ and the diagonal $d$}\label{Section d<}

In this section we derive restrictions on the pair $(n,d)$ for matrices satisfying~\eqref{Conditions}.
The symbol $\mathbb{N}_0$ used in the text denotes the set of non-negative integers.

\begin{prop}\label{Prop1.d}
If a matrix $C$ satisfies \eqref{Conditions}, then $2d$ is an integer. Moreover:
\begin{itemize}
\item[(i)] If $d$ is a half-integer, then $n=2d+2$.
\item[(ii)] If $d$ is an integer, then
\begin{equation}\label{d in k}
n=k(2d+k)+1
\end{equation}
for some odd $k\in\mathbb{N}$. In particular, $n$ is even, and $d$ is odd if and only if $\frac{n}{2}$ is even.
\end{itemize}
\end{prop}

\begin{proof}
For any $r\in\{1,\ldots,n-1\}$, the scalar product of the $0$th and $r$th row of $C$ must be zero; hence
\begin{equation}\label{orthog n k}
d(c_r+c_{n-r})=-\sum_{\substack{j=1\\j\neq n-r}}^{n-1}c_jc_{(j+r) \bmod n}\,.
\end{equation}
Since $c_j\in\{1,-1\}$ for all $j=1,\ldots,\frac{n}{2}-1$, the left hand side of \eqref{orthog n k} satisfies $d(c_r+c_{n-r})\in\{-2d,0,2d\}$, while the right hand side of \eqref{orthog n k} is always an integer of the same parity as $n$.

If $2d\notin\mathbb{N}_0$, the left hand side of $\eqref{orthog n k}$ is an integer only when being equal to $0$. Thus the right hand side must be $0$, too; hence $n$ is even. Then, however, \eqref{orthog n k} cannot be satisfied for $r=\frac{n}{2}$, because the left hand side is $2dc_{\frac{n}{2}}\notin\mathbb{Z}$. Consequently, $2d$ is an integer.

(i)\; Let $d$ be a half-integer, i.e., $2d$ is odd. Then $n$ is odd, otherwise \eqref{orthog n k} would be violated for $r=\frac{n}{2}$.
As a result, the left hand side of \eqref{orthog n k} must be odd for all $r=1,\ldots,n-1$, i.e., $c_r+c_{n-r}\neq0$ for all $r$. Considering that $c_j\in\{1,-1\}$, we conclude that $C$ is symmetric.

Let us take an arbitrary $r\in\{1,\ldots,n-1\}$ and denote $c_r=c_{n-r}=\gamma$. We write down the $0$th and $r$th row of $C$ and rearrange the columns in the following way:
\begin{equation*}
\begin{array}{ccccccl}
d & \gamma & +1 \cdots +1 & +1 \cdots +1 & -1 \cdots -1 & -1 \cdots -1 \\
\gamma & d & \underbrace{+1 \cdots +1}_{\ell_1} & \underbrace{-1 \cdots -1}_{\ell_2} & \underbrace{+1 \cdots +1}_{\ell_3} & \underbrace{-1 \cdots -1}_{\ell_4} & .
\end{array}
\end{equation*}
We have
\begin{equation}\label{odd L1}
\ell_1+\ell_2+\ell_3+\ell_4=n-2\,.
\end{equation}
Since $C$ is circulant, every row of $C$ has the same sum of elements, i.e.,
\begin{equation}\label{odd L3,4}
d+\gamma+\ell_1+\ell_2-\ell_3-\ell_4=\gamma+d+\ell_1-\ell_2+\ell_3-\ell_4\,.
\end{equation}
Since $C$ is orthogonal, the scalar product of the $0$th and the $r$th row must be $0$; hence
\begin{equation}\label{odd L2}
2\gamma d+\ell_1-\ell_2-\ell_3+\ell_4=0\,.
\end{equation}
The system of equations~\eqref{odd L1}--\eqref{odd L2} implies
$$
4\ell_2=n-2+2\gamma d\,.
$$
Consequently,
\begin{equation}\label{n odd mod}
n-2+2\gamma d\equiv0\pmod4\,.
\end{equation}

Obviously there is an $r$ such that $c_r=-1$; otherwise the rows of $C$ would not be orthogonal. We already know that $C$ is symmetric, hence $c_{n-r}=c_r=-1$. Setting $\gamma=-1$ in equation~\eqref{n odd mod}, we get
\begin{equation}\label{odd n minus}
n-2-2d\equiv0\pmod4\,.
\end{equation}
If there was also a $r'$ such that $c_{r'}=c_{n-r'}=+1$, then, with regard to \eqref{n odd mod}, we would have one more equation, namely,
$$
n-2+2d\equiv0\pmod4\,.
$$
This equation together with \eqref{odd n minus} implies $2n-4\equiv0\pmod4$, which is in contradiction with the fact that $n$ is odd. We conclude that $c_j=-1$ for all $j=1,\ldots,n-1$, i.e., the generator of $C$ is $(d,-1,-1,\ldots,-1)$. Equation~\eqref{orthog n k} then takes the form $-2d=n-2$; hence $n=2d+2$.

(ii)\; If $d$ is an integer, then $d(c_k+c_{n-k})\in\{-2d,0,2d\}$ is even; hence $n$ is even by \eqref{orthog n k}.
From equation~\eqref{sum same} we have
\begin{equation}\label{perf.square}
|d+c_1+\cdots+c_{n-1}|=\sqrt{d^2+n-1}\,.
\end{equation}
Since $c_j\in\{1,-1\}$ for all $j=1,\ldots,n-1$, the left hand side of \eqref{perf.square} is an integer. Therefore, there exists a $k\in\mathbb{Z}$ such that $|d+c_1+\cdots+c_{n-1}|=d+k$. Considering the right hand side of \eqref{perf.square}, $k$ is positive. So we have $d+k=\sqrt{d^2+n-1}$ for some $k\in\mathbb{N}$ (recall that $n>1$ by \eqref{Conditions}). Hence we obtain \eqref{d in k}. Since $n$ is even, equation \eqref{d in k} implies that $k$ must be odd. Finally, $d$ is odd if and only if $\frac{n}{2}$ is even, because $\frac{n}{2}-d=(k-1)d+\frac{k^2+1}{2}$ is obviously odd for every odd $k$.
\end{proof}

\begin{rem}\label{max d exists}
A matrix $C$ obeying conditions~\eqref{Conditions} exists for every $d\geq0$ such that $2d$ is an integer. For example, consider the generator $(d,-1,-1,\ldots,-1)\in\mathbb{R}^n$ for $n=2d+2$.
In particular, if $d$ is a half-integer, it immediately follows from the proof of Proposition~\ref{Prop1.d}(i) that $(d,-1,-1,\ldots,-1)\in\mathbb{R}^{2d+2}$ is the only possible generator of $C$.
\end{rem}

\begin{corol}\label{Coro.n}
If a matrix $C$ of order $n$ satisfies \eqref{Conditions} and $n-1$ is prime, then $d=\frac{n}{2}-1$.
\end{corol}

\begin{proof}
Proposition~\ref{Prop1.d} implies that $d$ is integer if and only if $n$ is even.
If $n-1=2$, then $n=3$ is odd, thus $d$ is a half-integer, and $n=2d+2$ due to Proposition~\ref{Prop1.d}.
If $n-1$ is an odd prime, then $n$ is even, hence $d\in\mathbb{N}_0$. Then equation~\eqref{d in k} gives $d=\frac{1}{2}\left(\frac{n-1}{k}-k\right)$ for some $k$ that divides $n-1$. Since $n-1$ is prime, we have $k=1$ (the other possibility, $k=n-1$, forces $d<0$, which is ruled out by~\eqref{Conditions}); hence $d=\frac{1}{2}\left(\frac{n-1}{1}-1\right)=\frac{n}{2}-1$.
\end{proof}
Note that the statement of Corollary~\ref{Coro.n} can be extended. One can show in a similar way that if $n-1$ is the square of a prime, then either $d=\frac{n}{2}-1$ or $d=0$, and if $n-1$ is the product of two twin primes, then either $d=\frac{n}{2}-1$ or $d=1$.

\begin{prop}\label{Prop.even}
If $C$ satisfies \eqref{Conditions} and $d$ is even, then
\begin{itemize}
\item[(i)] $C$ is symmetric;
\item[(ii)] $d\equiv\frac{n}{2}-1 \pmod 4$;
\item[(iii)] the entries of the generator obey equation $\left(\sum_{j=1}^{\frac{n}{2}}c_{2j-1}\right)^2=d^2+n-1$.
\end{itemize}
\end{prop}

\begin{proof}
If $d$ is even, Proposition~\ref{Prop1.d} implies $n\equiv2 \pmod 4$.

(i)\; We prove the statement by contradiction. Assume that $C$ is not symmetric. Then there exists a $j\in\{1,\ldots,n-1\}$ such that $c_j=1$ and $c_{n-j}=-1$. We write down the $0$th row and the $j$th row of $C$ and rearrange the columns as follows.
\begin{equation*}
\begin{array}{ccccccl}
d & +1 & +1 \cdots +1 & +1 \cdots +1 & -1 \cdots -1 & -1 \cdots -1 \\
-1 & d & \underbrace{+1 \cdots +1}_{\ell_1} & \underbrace{-1 \cdots -1}_{\ell_2} & \underbrace{+1 \cdots +1}_{\ell_3} & \underbrace{-1 \cdots -1}_{\ell_4} & .
\end{array}
\end{equation*}
Similarly as in the proof of Proposition~\ref{Prop1.d}(i), we use the properties of $C$ (order $n$, circulancy and orthogonality of rows) to obtain the conditions
\begin{equation}\label{L1}
\ell_1+\ell_2+\ell_3+\ell_4=n-2\,;
\end{equation}
\begin{equation}\label{L3,4}
d+1+\ell_1+\ell_2-\ell_3-\ell_4=-1+d+\ell_1-\ell_2+\ell_3-\ell_4\,;
\end{equation}
\begin{equation}\label{L2}
\ell_1-\ell_2-\ell_3+\ell_4=0\,.
\end{equation}
Solving the system of equations~\eqref{L1}--\eqref{L2}, we get in particular
$$
\ell_2=\frac{n}{4}-1\,.
$$
Consequently, $n$ is a multiple of $4$, which contradicts the above-mentioned relation $n\equiv2 \pmod 4$.

(ii)\; The symmetry of $C$ implies that the $0$th and the $\frac{n}{2}$th row of $C$ read
$$
\begin{array}{ccccccccccl}
d & c_{1} & c_{2} & \cdots & c_{\frac{n}{2}-1} & c_{\frac{n}{2}} & c_{\frac{n}{2}-1} & \cdots & c_2 & c_1 \\
c_{\frac{n}{2}} & c_{\frac{n}{2}-1} & c_{\frac{n}{2}-2} & \cdots & c_{1} & d & c_{1} & \cdots & c_{\frac{n}{2}-2} & c_{\frac{n}{2}-1} & .
\end{array}
$$
The two rows are orthogonal, i.e.,
$$
2d c_{\frac{n}{2}}+2\sum_{j=1}^{\frac{n}{2}-1}c_j c_{\frac{n}{2}-j}=0\,.
$$
Since $|c_{\frac{n}{2}}|=1$ and $d\geq 0$, we get
\begin{equation}\label{d 2 mod 4}
d=\left|\sum_{j=1}^{\frac{n}{2}-1}c_j c_{\frac{n}{2}-j}\right|\,.
\end{equation}
The relation $n\equiv2 \pmod 4$ implies
$$
\sum_{j=1}^{\frac{n}{2}-1}c_j c_{\frac{n}{2}-j}=2\sum_{j=1}^{\frac{n-2}{4}}c_j c_{\frac{n}{2}-j}\,,
$$
which allows us to rewrite equation~\eqref{d 2 mod 4} in the form
\begin{equation}\label{congr}
\frac{d}{2}=\left|\sum_{j=1}^{\frac{n-2}{4}}c_j c_{\frac{n}{2}-j}\right|\,.
\end{equation}
Since the sum on the right hand side of \eqref{congr} has the same parity as the number $\frac{n-2}{4}$, we have $\frac{d}{2}\equiv\frac{n-2}{4} \pmod 2$. And so $d\equiv\frac{n}{2}-1 \pmod 4$.

(iii)\; Let us denote $a=\sum_{j=0}^{\frac{n}{2}-1}c_{2j}$, $b=\sum_{j=1}^{\frac{n}{2}}c_{2j-1}$.
Equations~\eqref{sum same} and \eqref{sum alternating} give
$$
(a+b)^2=d^2+n-1\,, \qquad (a-b)^2=d^2+n-1\,.
$$
Thus $(a+b)^2=(a-b)^2$, which implies $ab=0$.
Since $\frac{n}{2}$ is odd (Proposition~\ref{Prop1.d}) and $b$ consists of $\frac{n}{2}$ terms $\pm1$, $b\neq0$. Hence $a=0$.
\end{proof}

Now we are ready to solve the case when $d$ is an even integer. Theorem~\ref{Thm: d even} below generalizes a theorem of Stanton and Mullin~\cite{SM} which says that a circulant conference matrix exists only for $n=2$. The idea of the proof is based on~\cite{SM}.

\begin{thm}\label{Thm: d even}
If a matrix $C$ satisfies conditions~\eqref{Conditions} and $d$ is an even integer, then $n=2d+2$.
\end{thm}

\begin{proof}
The even parity of $d$ implies that $\frac{n}{2}$ is an odd integer (Proposition~\ref{Prop1.d}). Since $C$ is symmetric due to Proposition~\ref{Prop.even}, its $0$th row and the $\ell$th row for $\ell\in\left\{1,\ldots,\frac{n}{2}-1\right\}$ take the form
$$
\begin{array}{ccccccccccccccl}
d & c_{1} & \cdots & c_{\ell-1} & c_{\ell} & c_{\ell+1} & \cdots & c_{\frac{n}{2}} & c_{\frac{n}{2}-1} & \cdots & c_{\frac{n}{2}-\ell+1} & c_{\frac{n}{2}-\ell} & \cdots & c_1 \\
c_\ell & c_{\ell-1} & \cdots & c_1 & d & c_{1} & \cdots & c_{\frac{n}{2}-\ell} & c_{\frac{n}{2}-\ell+1} & \cdots & c_{\frac{n}{2}-1} & c_{\frac{n}{2}} & \cdots & c_{\ell+1} & .
\end{array}
$$
Their scalar product shall be zero, i.e.,
\begin{equation}\label{SPell}
2dc_\ell+\sum_{j=1}^{\ell-1}c_jc_{\ell-j}+2\sum_{j=1}^{\frac{n}{2}-\ell}c_jc_{j+\ell}+\sum_{j=\frac{n}{2}-\ell+1}^{\frac{n}{2}-1}c_{j}c_{n-\ell-j}=0\,.
\end{equation}
From now on let $\ell$ be odd. We have $\ell=2h+1$ for some $h$, and
\begin{gather*}
\sum_{j=1}^{\ell-1}c_jc_{\ell-j}=2\sum_{j=1}^{h}c_jc_{\ell-j}\,; \\
\sum_{j=\frac{n}{2}-\ell+1}^{\frac{n}{2}-1}c_{j}c_{n-\ell-j}=2\sum_{j=\frac{n}{2}-h}^{\frac{n}{2}-1}c_{j}c_{n-\ell-j}\,.
\end{gather*}
With regard to these two identities, equation~\eqref{SPell} implies
\begin{equation}\label{dS}
d=\left|\sum_{j=1}^{\frac{n}{2}-\ell}c_jc_{\ell+j}+\sum_{j=1}^{h}c_jc_{\ell-j}+\sum_{j=\frac{n}{2}-h}^{\frac{n}{2}-1}c_{j}c_{n-\ell-j}\right|\,.
\end{equation}
Let us denote the sum appearing on the right hand side of equation~\eqref{dS} by $S$, i.e.,
\begin{equation}\label{S}
S:=\underbrace{\sum_{j=1}^{\frac{n}{2}-\ell}c_jc_{\ell+j}}_{S_1}+\underbrace{\sum_{j=1}^{h}c_jc_{\ell-j}}_{S_2}+\underbrace{\sum_{j=\frac{n}{2}-h}^{\frac{n}{2}-1}c_{j}c_{n-\ell-j}}_{S_3}\,.
\end{equation}
The sum $S$ consists of products $c_ic_j$ for $i,j\in\{1,\ldots,\frac{n}{2}\}$. It is easy to see that each term $c_ic_j$ for $i,j\in\{1,\ldots,\frac{n}{2}\}$ occurs at most once in $S$. Let us define a graph $G=(V,E)$ with the set of vertices $V=\{1,\ldots,\frac{n}{2}\}$ and the set of edges $E$ given by the following condition: $\{i,j\}\in E$ if and only if $c_ic_j$ is a summand of $S$.
Let us show that vertices $\ell$ and $\frac{n}{2}$ of $G$ have degree $1$ and all others have degree $2$. We distinguish two cases. For $\ell<\frac{n+2}{4}$, we have:
\begin{itemize}
\item If $j\in[1,\ell-1]$, then the factor $c_j$ occurs once in $S_1$ (in the product $c_jc_{j+\ell}$) and once in $S_2$ (in the product $c_jc_{\ell-j}$). Recall that the quantity $h$ appearing in $S_2$ and $S_3$ was introduced by the relation $\ell=2h+1$.
\item If $j\in[\ell+1,\frac{n}{2}-\ell]$, then the factor $c_j$ occurs only in summands of $S_1$, namely, in the products $c_jc_{j+\ell}$ and $c_{j-\ell}c_j$.
\item If $j\in[\frac{n}{2}-\ell+1,\frac{n}{2}-1]$, then the factor $c_j$ occurs once in $S_1$ (in the product $c_{j-\ell}c_j$) and once in $S_3$ (in the product $c_jc_{n-\ell-j}$).
\item The factor $c_\ell$ occurs only in the sum $S_1$, namely, in the product $c_\ell c_{2\ell}$.
\item The factor $c_\frac{n}{2}$ occurs only in the sum $S_1$, namely, in the product $c_{\frac{n}{2}-\ell}c_\frac{n}{2}$.
\end{itemize}
Case $\ell\geq\frac{n+2}{4}$ is similar:
\begin{itemize}
\item If $j\in[1,\frac{n}{2}-\ell]$, the factor $c_j$ occurs once in $S_1$ and once in $S_2$;
\item if $j\in[\frac{n}{2}-\ell+1,\ell-1]$, the factor $c_j$ occurs once in $S_2$ and once in $S_3$;
\item the factor $c_\ell$ occurs only once in $S_3$;
\item if $j\in[\ell+1,\frac{n}{2}-1]$, the factor $c_j$ occurs once in $S_1$ and once in $S_3$;
\item the factor $c_\frac{n}{2}$ occurs only once in $S_1$.
\end{itemize}
Consequently, the graph $G$ consists of connected components of two types:
\begin{itemize}
\item a simple path $P=(v_0,v_1,\ldots,v_{L})$ with $v_0=\ell$ and $v_{L}=\frac{n}{2}$;
\item a certain number (possibly zero) of simple cycles $R_k=(v^{(k)}_0,v^{(k)}_1,\ldots,v^{(k)}_{L_k})$ with $v^{(k)}_0=v^{(k)}_{L_k}$, where $k\in K$. If the graph $G$ is connected, then $G$ consists of the simple path $P$ and the set $K$ is empty.
\end{itemize}
The lengths $L$ and $L_k$, as well as the cardinality of $K$, are not important for our considerations.

Since the products $c_ic_j$ occurring as summands of $S$ represent the edges of $G$, we can rearrange them to follow the order of edges on the path $P$ and on the cycles $R_k$,
\begin{equation}\label{SS}
S=\sum_{i=0}^{L-1}c_{v_i}c_{v_{i+1}}+\sum_{k\in K}\sum_{i=0}^{L_k-1}c_{v^{(k)}_i}c_{v^{(k)}_{i+1}}\,.
\end{equation}

The sum $S$ contains $\frac{n}{2}-1$ terms of type $\pm1$ by \eqref{S}. Therefore, $S=\frac{n}{2}-1-2s$, where $s$ is the total number of negative summands in $S$.
Equation~\eqref{dS} says that $d=|S|$, i.e.,
\begin{equation}\label{d|S|}
d=\left|\frac{n}{2}-1-2s\right|\,.
\end{equation}
The left hand side of \eqref{d|S|} satisfies $d\equiv\frac{n}{2}-1\pmod4$ according to Proposition~\ref{Prop.even}.
The right hand side of \eqref{d|S|} must be an even integer (because the left hand side is even by assumption); hence we get $|\frac{n}{2}-1-2s|\equiv\frac{n}{2}-1-2s\pmod4$.
Combining these two facts, we obtain $\frac{n}{2}-1\equiv\frac{n}{2}-1-2s\pmod4$, i.e., $2s\equiv0\pmod4$. This means that $s$ is even, i.e., the sum $S$ must contain an even number of negative summands.

Equation~\eqref{SS} implies that the number of negative summands in $S$ is equal to the number of sign changes in the sequence $c_{v_0},\ldots,c_{v_{L}}$ plus the number of sign changes in all the sequences $c_{v^{(k)}_0},\ldots,c_{v^{(k)}_{L_k}}$ for $k\in K$. Since $v^{(k)}_0=v^{(k)}_{L_k}$ for each $k$ (recall that $R_k$ is a cycle), each sequence $c_{v^{(k)}_0},\ldots,c_{v^{(k)}_{L_k}}$ contains an even number of sign changes. Therefore, there must be an even number of sign changes in the sequence $c_{v_0},\ldots,c_{v_{L}}$ as well; hence $c_{v_0}=c_{v_{L}}$. We have $v_0=\ell$ and $v_{L}=\frac{n}{2}$, whence we get the condition
\begin{equation}\label{c_ell}
c_\ell=c_{\frac{n}{2}}\,.
\end{equation}
Equation~\eqref{c_ell} is valid for any odd number $\ell=1,3,\ldots,\frac{n}{2}-2$. The symmetry of $C$ means $c_i=c_{n-i}$ for all $i=1,\ldots,\frac{n}{2}$; therefore, \eqref{c_ell} is satisfied also for $\ell=\frac{n}{2}+2,\ldots,n-3,n-1$. Consequently,
\begin{equation}\label{sum odd 1}
\sum_{j=1}^{\frac{n}{2}}c_{2j-1}=\frac{n}{2}c_{\frac{n}{2}}\,.
\end{equation}
At the same time we have, due to Proposition~\ref{Prop.even},
\begin{equation}\label{sum odd 2}
\left(\sum_{j=1}^{\frac{n}{2}}c_{2j-1}\right)^2=d^2+n-1\,.
\end{equation}
Equations~\eqref{sum odd 1} and \eqref{sum odd 2} imply $d^2+n-1=\left(\frac{n}{2}\right)^2$; hence $d=\frac{n}{2}-1$.
\end{proof}

Using Proposition~\ref{Prop1.d}, Corollary~\ref{Coro.n} and Theorem~\ref{Thm: d even}, we can immediately disprove the existence of matrices $C$ with the property $n\neq2d+2$ for all pairs $(n,d)$ up to the order $n=50$ with the following $4$ exceptions:
$$
(16,1),\; (28,3),\; (36,1),\; (40,5)\,.
$$
A computer calculation confirmed that there is no solution for any of the pairs $(n,d)$ in the above list. In other words, up to the order $n=50$ all matrices $C$ obeying \eqref{Conditions} have the property $n=2d+2$.
Our findings lead us to proposing the following conjecture.

\begin{conj}\label{Conjecture}
A circulant matrix $C$ of order $n\geq2$ having the generator $(d,c_1,\ldots,c_{n-1})$ with $d\geq 0$ and $c_j\in\{1,-1\}$ for all $j=1,\ldots,n-1$ satisfies the condition $CC^T=(d^2+n-1)I$ only if $n=2d+2$.
\end{conj}

\begin{rem}\label{Rem.conj}
Let us summarize facts concerning the validity of Conjecture~\ref{Conjecture}.
\begin{itemize}
\item We have established the conjecture for all cases where $d$ is not an odd integer.
\item As a result of performed computer calculations, the conjecture is confirmed for matrices $C$ of orders up to $n=50$.
\item Conjecture~\ref{Conjecture} generalizes the circulant Hadamard conjecture, which corresponds to $d=1$.
\end{itemize}
\end{rem}

\section{Symmetric solutions}\label{Section: symmetric}

In this section we generalize the well-known result about the nonexistence of symmetric circulant Hadamard matrices of order $n>4$ by proving that if a matrix $C$ satisfying conditions~\eqref{Conditions} is symmetric, then $n=2d+2$.

\begin{prop}\label{Prop.symmetric}
If a matrix $C$ satisfies \eqref{Conditions} for an odd $d$ and $C$ is symmetric, then $d^2-1$ is divisible by $2\sqrt{d^2+n-1}$.
\end{prop}

\begin{proof}
Since $d$ is an odd integer, $n$ is even due to Proposition~\ref{Prop1.d}. Equation~\eqref{sum same} then implies that $\sqrt{d^2+n-1}$ is an even integer; let us denote this integer by $\ell$. If $C$ is symmetric, it has a generator $(d,c_1,\ldots,c_{\frac{n}{2}-1},c_{\frac{n}{2}},c_{\frac{n}{2}-1},\ldots,c_1)$. Following an idea from~\cite[proof of Theorem~8]{Cr}, let us consider a symmetric circulant matrix $M$ with the generator $(c_{\frac{n}{2}},c_{\frac{n}{2}-1},\ldots,c_1,d,c_1,\ldots,c_{\frac{n}{2}-1})$. Since $CC^T=(d^2+n-1)I$ and $M=PC$ for some permutation matrix $P$, we have $MM^T=(d^2+n-1)I$. Therefore, $M$ has eigenvalues $\pm \ell$ for $\ell=\sqrt{d^2+n-1}$. If we denote the multiplicity of the eigenvalue $+\ell$ of $M$ by $m$, the sum of eigenvalues of $M$ is $2\left(m-\frac{n}{2}\right)\ell$. At the same time the sum of eigenvalues of $M$ is equal to $\mathrm{Tr}(M)=nc_{\frac{n}{2}}$. Comparing these quantities, we obtain $2\ell\mid n$. Now we express $n$ in terms of $\ell$, i.e., $n=\ell^2+1-d^2$. Since $\ell$ is even, we have $2\ell\mid \ell^2$. This allows us to transform the condition $2\ell\mid(\ell^2+1-d^2)$ into $2\ell\mid(d^2-1)$.
\end{proof}

\begin{ex}
Proposition~\ref{Prop.symmetric} implies that a symmetric matrix $C$ satisfying~\eqref{Conditions} with $d=3$ exists only for $n=8$. Indeed, $2\sqrt{3^2+n-1}\mid(3^2-1)$ requires $\sqrt{8+n}=4$; hence $n=8$.
\end{ex}

In~\cite{MKW}, McKay and Wang found a strong inequality between the order $n$ of a symmetric circulant Hadamard matrix and the prime factorization of $n$, and used it for disproving the existence of symmetric Hadamard matrices of order $n>4$. Taking advantage of their idea, we derive a similar inequality for matrices $C$ with a general $d\in\mathbb{N}$ that relates the prime factorization of $n$ to the integer $k$ appearing in formula~\eqref{d in k}.

\begin{prop}\label{Prop. factorization}
Let a symmetric matrix $C$ satisfy \eqref{Conditions} with $d\in\mathbb{N}$ and $n=k(2d+k)+1$ for some odd $k\in\mathbb{N}$. Let $n=q_1^{\alpha_1}q_2^{\alpha_2}\cdots q_r^{\alpha_r}$ be the prime factorization of $n$. Then
\begin{equation}\label{k r}
k+1\leq2^r\,.
\end{equation}
\end{prop}

\begin{proof}
We will proceed in a similar way as McKay and Wang did in~\cite[Proof of Theorem~3]{MKW}, with some modifications that are required with regard to the generality of $d$.
The first step consists in proving that
\begin{equation}\label{gcd}
\gcd(j,n)=m \quad\text{implies}\quad c_j=c_m\,.
\end{equation}
For each $m\mid n$ we define the polynomial
$$
P_m(x)=c_0+c_{1}x+c_{2}x^{2}+\cdots+c_{n-1}x^{n-1}-\lambda_m\,,
$$
where each $\lambda_m$ is given by \eqref{eigenvalues}.
Since $C=C^T$ and $CC^T=(d^2+n-1)I$, the eigenvalues of $C$ satisfy $\lambda_m=\pm\sqrt{d^2+n-1}$. The assumption $n=k(2d+k)+1$ for some $k\in\mathbb{N}$ then gives $\lambda_m=\pm(d+k)\in\mathbb{Z}$ for all $m$. Therefore, the polynomial $P_m(x)$ has integer coefficients for each $m$. Furthermore, $P_m(\omega^m)=\lambda_m-\lambda_m=0$ for every $m$, where $\omega=\mathrm{e}^{2\pi\mathrm{i}/n}$.

Let $\Phi_N(x)$ denote the $N$th cyclotomic polynomial. Then $\Phi_{N}(\mathrm{e}^{2\pi\mathrm{i}K/N})=0$ for every $K\in\{0,1,\ldots,N-1\}$ satisfying $\gcd(K,N)=1$. If we set $N=\frac{n}{m}$ and $K=1$, we get $\Phi_{\frac{n}{m}}(\omega^m)=0$. Since $P_m(\omega^m)=0$ and $\Phi_N(x)$ is irreducible by definition, necessarily $\Phi_{\frac{n}{m}}(x)\mid P_m(x)$.

The fact $\Phi_{\frac{n}{m}}(x)\mid P_m(x)$ implies that $P_m(x)=0$ whenever $\Phi_{\frac{n}{m}}(x)=0$. From now on let $\gcd(j,n)=m$. If we set $N=\frac{n}{m}$ and $K=\frac{j}{m}$, we have $\gcd(K,N)=\frac{1}{m}\gcd(j,n)=1$. Therefore, $\Phi_{\frac{n}{m}}(\mathrm{e}^{2\pi\mathrm{i}j/n})=0$. Hence we infer $P_m(\mathrm{e}^{2\pi\mathrm{i}j/n})=0$. This means $\lambda_j-\lambda_m=0$, i.e., $\lambda_j=\lambda_m$.

Using formula~\eqref{invDFT} and the result $\lambda_j=\lambda_m$ for $\gcd(j,n)=m$, we can express $c_j$ in the form
\begin{equation}\label{cj}
c_j=\frac{1}{n}\left[\lambda_0+\sum_{\substack{h\mid n \\ 1\leq h\leq n-1}}\lambda_h\left(\sum_{\substack{\gcd(\ell,n)=h \\ 1\leq \ell<n}}\omega^{-\ell j}\right)\right]\,.
\end{equation}
If $\gcd(j,n)=m$, we have $j=Km$ for some $K$ such that $\gcd(K,n)=1$. Then
$$
\sum_{\substack{\gcd(\ell,n)=h \\ 1\leq \ell<n}}\omega^{-\ell j}=
\sum_{\substack{\gcd(\frac{\ell}{h},\frac{n}{h})=1 \\ 1\leq\frac{\ell}{h}<\frac{n}{h}}}\left(\omega^{hm}\right)^{-K\frac{\ell}{h}}=
\sum_{\substack{\gcd(\ell',\frac{n}{h})=1 \\ 1\leq \ell'<\frac{n}{h}}}\left(\omega^{hm}\right)^{-K\ell'}=
\sum_{\substack{\gcd(\ell'',\frac{n}{h})=1 \\ 1\leq \ell''<\frac{n}{h}}}\left(\omega^{hm}\right)^{-\ell''}=
\sum_{\substack{\gcd(\ell,n)=h \\ 1\leq \ell<n}}\omega^{-\ell m}\,,
$$
for every $h$ dividing $n$, where we used the fact that $\gcd(K\ell',\frac{n}{h})=1$ if and only if $\gcd(\ell',\frac{n}{h})=1$,
which follows from $\gcd(K,\frac{n}{h})=1$.
Considering~\eqref{cj}, we conclude: If $\gcd(j,n)=m$, then $c_j=c_m$.

Now we can proceed to the second step. Equation~\eqref{eigenvalues} together with \eqref{gcd} allows us to express the eigenvalue $\lambda_1$ of $C$ in the form
$$
\lambda_1=c_0+\sum_{j=1}^{n-1}c_j\omega^j=c_0+\sum_{\substack{m\mid n \\ 1\leq m\leq n-1}}c_m\left(\sum_{\substack{\gcd(j,n)=m \\ 1\leq j\leq n-1}}\omega^j\right)\,,
$$
where $c_0=d$. We have
$$
\sum_{\substack{\gcd(j,n)=m \\ 1\leq j\leq n-1}}\omega^j=\sum_{\substack{\gcd(q,\frac{n}{m})=1 \\ 1\leq q\leq \frac{n}{m}-1}}(\omega^m)^q\,,
$$
which is the sum of primitive $\frac{n}{m}$th roots of unity. According to a classical formula \cite[(16.6.4)]{HW}, this sum is equal to $\mu(\frac{n}{m})$, where $\mu$ is the M\"{o}bius function. Therefore,
$$
\lambda_1=d+\sum_{\substack{m\mid n \\ 1\leq m\leq n-1}}c_m\mu\left(\frac{n}{m}\right)\,.
$$
Since $\mu(1)=1$, we can rewrite the equation in the form
\begin{equation}\label{lambda1}
\lambda_1=d-1+\mu(1)+\sum_{\substack{m\mid n \\ 1\leq m\leq n-1}}c_m\mu\left(\frac{n}{m}\right)\,.
\end{equation}
We have $|\lambda_1|=d+k$, $d\in\mathbb{N}$ and $|c_j|=1$ for all $j\geq1$. Therefore, equation~\eqref{lambda1} implies
$$
d+k\leq d-1+\sum_{\substack{m\mid n \\ 1\leq m\leq n}}\left|\mu\left(\frac{n}{m}\right)\right|=d-1+\sum_{\substack{m\mid n \\ 1\leq m\leq n}}|\mu(m)|\,.
$$
Hence
\begin{equation}\label{|lambda1|}
k+1\leq\sum_{\substack{m\mid n \\ 1\leq m\leq n}}|\mu(m)|\,.
\end{equation}
Let $n=q_1^{\alpha_1}q_2^{\alpha_2}\cdots q_r^{\alpha_r}$ be the prime factorization of $n$. By definition of $\mu$, we have
$$
|\mu(\ell)|=\begin{cases}
1, & \text{if $\ell$ is a square-free positive integer;} \\
0, & \text{if $\ell$ has a squared prime factor.}
\end{cases}
$$
Therefore, if $n=q_1^{\alpha_1}q_2^{\alpha_2}\cdots q_r^{\alpha_r}$, the sum on the right hand side of inequality~\eqref{|lambda1|} is equal to the number of subsets of $\{q_1,\ldots,q_r\}$, i.e., to $2^r$. Hence we obtain inequality~\eqref{k r}.
\end{proof}

\begin{rem}
The inequality of McKay and Wang, derived for $d=1$ and $n>1$, reads $\sqrt{n}\leq2^r$.
\end{rem}

Technical Lemma~\ref{Lemma bL} below contains a result that will be used twice in the sequel. At first, it will allow us to estimate $n$ in the proof of Proposition~\ref{Prop. k large}. Secondly, it will be crucial for reducing the proof of Proposition~\ref{Prop. k small} to an examination of a finite number of cases.

Since the existence of matrices $C$ satisfying \eqref{Conditions} for $n\neq2d+2$ is impossible for $d\notin\mathbb{N}_0$ or $d$ being even (Proposition~\ref{Prop1.d} and Theorem~\ref{Thm: d even}), we may assume that $d$ is odd.

\begin{lem}\label{Lemma bL}
Let a symmetric matrix $C$ satisfy \eqref{Conditions} for an odd $d$, and let $n=k(2d+k)+1$ for some odd $k>1$. Then there exist $t,u,w,z\in\mathbb{N}$ such that $w<t$ and
\begin{equation}\label{bL}
\frac{k+1}{2}=tu\,, \quad \frac{k-1}{2}=wz \quad\text{and}\quad n=4tz(2tu-1-uw)\,.
\end{equation}
\end{lem}

\begin{proof}
Since $k$ is odd, we have $\frac{k+1}{2}\in\mathbb{N}$, $\frac{k-1}{2}\in\mathbb{N}$ and $\frac{d+k}{2}\in\mathbb{N}$.
We set
\begin{equation}\label{s/t}
\frac{d+1}{d+k}=\frac{s}{t} \quad\text{ for }\; s,t\in\mathbb{N},\; \gcd(s,t)=1\,.
\end{equation}
With regard to the assumption $k>1$, we have $s<t$.
Equation~\eqref{s/t} implies $\frac{d+k}{2}=\frac{t}{s}\cdot\frac{d+1}{2}$. Since $\gcd(s,t)=1$, we have $s\mid\frac{d+1}{2}$, i.e., $\frac{d+1}{2}=zs$ for some $z\in\mathbb{N}$.
Then
$$
\frac{d+k}{2}=\frac{t}{s}\cdot\frac{d+1}{2}=tz\,.
$$
According to Proposition~\ref{Prop.symmetric}, we have $2(d+k)\mid(d^2-1)$. Therefore, $\frac{(d+1)(d-1)}{2(d+k)}=\frac{s}{t}\cdot\frac{d-1}{2}\in\mathbb{N}$. We use again the assumption $\gcd(s,t)=1$ to infer that $\frac{d-1}{2}=vt$ for some $v\in\mathbb{N}$.
Hence we get
\begin{gather*}
\frac{k+1}{2}=\frac{d+k}{2}-\frac{d-1}{2}=tz-vt=t(z-v)\,;
\\
\frac{k-1}{2}=\frac{d+k}{2}-\frac{d+1}{2}=tz-zs=z(t-s)\,.
\end{gather*}
If we set $z-v=:u$ and $t-s=:w$, we get $\frac{k+1}{2}=tu$ and $\frac{k-1}{2}=zw$.
It remains to express $n$ in terms of $t,u,w,z$. For this purpose we rewrite
$$
n=k(2d+k)+1=2(d+k)(k+1)-2(d+k)-k^2+1=(d+k)\left(2(k+1)-2-\frac{(k+1)(k-1)}{d+k}\right)
$$
and take advantage of equations $k+1=2tu$, $k-1=2wz$ and $d+k=2tz$ derived above. This gives
\begin{equation*}
n=2tz\left(4tu-2-\frac{2tu\cdot2zw}{2tz}\right)=4tz(2tu-1-uw)\,.
\end{equation*}
\end{proof}

\begin{prop}\label{Prop. k large}
Let $d$ be odd and $n=k(2d+k)+1$ for an odd $k$. If $k\geq 2^7$, then a symmetric matrix $C$ satisfying \eqref{Conditions} does not exist.
\end{prop}

\begin{proof}
Let $n=q_1^{\alpha_1}q_2^{\alpha_2}\cdots q_r^{\alpha_r}$ be the prime factorization of $n$.
If $r\leq7$, then $k+1>k\geq2^7\geq2^r$, and the statement follows straightforwardly from Proposition~\ref{Prop. factorization}.
So let $r\geq8$.
According to Lemma~\ref{Lemma bL}, values $n$ and $k$ satisfy equations~\eqref{bL}. In particular, we have
$$
tu=wz+1>z\,;
$$
hence
\begin{equation}\label{n k+1}
n=4tz(2tu-1-uw)<4t\cdot tu\cdot2tu=8t^3u^2\leq8t^3u^3=(2tu)^3=(k+1)^3\,.
\end{equation}
Since $d$ is odd, $n$ is a multiple of $4$ due to Proposition~\ref{Prop1.d}. Therefore, $q_1=2$ and $\alpha_1\geq2$. Then
\begin{equation}\label{n pr}
n\geq 2^2q_2q_3\cdots q_r\geq2p_r\#\,,
\end{equation}
where $p_r\#=\prod_{j=1}^r{p_j}=2\cdot3\cdot5\cdots p_r$ is the $r$th primorial number (the product of the first $r$ primes).
We have
\begin{equation}\label{primorial}
p_r\#>\frac{8^r}{2} \quad \text{for all $r\geq8$}\,,
\end{equation}
which follows from the fact that $p_8\#=9699690$, $\frac{8^3}{2}=8388608$ and $p_j>8$ for all $j>8$.
When we combine inequalities~\eqref{n k+1}, \eqref{n pr} and~\eqref{primorial}, we get
$$
k+1>\sqrt[3]{2p_r\#}>2^r \quad \text{for all $r\geq8$}\,,
$$
and the statement again follows from Proposition~\ref{Prop. factorization}.
\end{proof}

\begin{prop}\label{Prop. k small}
Let $d$ be odd and $n=k(2d+k)+1$ for an odd $k$. If $1<k\leq 2^7$, then a symmetric matrix $C$ of order $n$ satisfying \eqref{Conditions} does not exist with possible exceptions for $k=7,n=120$ and $k=13,n=924$.
\end{prop}

\begin{proof}
Our strategy is to verify that for every odd $k\leq2^7$ and for every $n=k(2d+k)+1$ allowed by Lemma~\ref{Lemma bL}, except for $k=7,n=120$ and $k=13,n=924$, we have $k+1>2^r$, where $q_1^{\alpha_1}q_2^{\alpha_2}\cdots q_r^{\alpha_r}$ is the prime factorization of $n$. Then the statement follows from Proposition~\ref{Prop. factorization}.

The verification is done step by step for each $k=3,5,7,\ldots,2^7-1$ using the following procedure, which is based on system~\eqref{bL}.
\begin{enumerate}
\item Find all possible $4$-tuples $(t,u,w,z)\in\mathbb{N}^4$ such that $\frac{k+1}{2}=tu$ and $\frac{k-1}{2}=wz$ with $w<t$.
\item For each $(t,u,w,z)$, set $n=4tz(2tu-1-uw)$ and find the prime factorization $n=q_1^{\alpha_1}q_2^{\alpha_2}\cdots q_r^{\alpha_r}$.
\item Check the inequality $k+1>2^r$ for all values $r$ found in the previous step.
\end{enumerate}

Let us demonstrate the procedure for $k=3,5,7$.

\begin{itemize}
\item Let $k=3$, i.e., $\frac{k+1}{2}=2$. Step 1: The system $tu=2$, $wz=1$, $w<t$ implies $t=2$, $u=1$, $w=1$, $z=1$. Step 2: $n=4\cdot2(2\cdot2-1-1)=16=2^4$; hence $r=1$. Step 3: We have $3+1>2^1$.

\item Let $k=5$. Step 1: $tu=3$, $wz=2$, $w<t$ implies $(t,u,w,z)\in\{(3,1,2,1),(3,1,1,2)\}$. Step 2: For $(3,1,2,1)$ we get $n=12\cdot3=2^2\cdot3^2$; hence $r=2$. For $(3,1,1,2)$ we get $n=24\cdot4=2^5\cdot3$; hence $r=2$. Step 3: In both cases we have $5+1>2^2$.

\item Let $k=7$. Step 1: $tu=4$, $wz=3$, $w<t$ implies $(t,u,w,z)\in\{(4,1,3,1),(4,1,1,3),(2,2,1,3)\}$. Step 2: For $(4,1,3,1)$ we get $n=16\cdot4=2^6$; hence $r=1$. For $(4,1,1,3)$ we get $n=48\cdot6=2^5\cdot3^2$; hence $r=2$. For $(2,2,1,3)$ we get $n=24\cdot5=2^3\cdot3\cdot5$; hence $r=3$. Step 3: If $r=1$ or $r=2$, then $7+1>2^r$. However, if $r=3$, we have $7+1=2^r$, i.e., $7+1\not>2^r$. Case $r=3$ occurs for
$$
n=120\,, \quad d=\frac{1}{2}\left(\frac{n-1}{k}-k\right)=\frac{1}{2}\left(\frac{119}{7}-7\right)=5\,.
$$
\end{itemize}

The calculation is straightforward and can be carried out completely with pen and paper, or on a computer, which gives results immediately. One finds that the inequality $k+1>2^r$ is satisfied for all remaining odd values $9\leq k\leq127$ except for $k=13$ with $(t,u,w,z)=(7,1,2,3)$. In this case we have $n=924=2^2\cdot3\cdot7\cdot11$, thus $r=4$, and $k+1=14\not>2^r$. The corresponding value of $d$ is $d=\frac{1}{2}\left(\frac{924-1}{13}-13\right)=29$.
\end{proof}

\begin{prop}\label{5,29}
There exists no symmetric matrix $C$ satisfying \eqref{Conditions} for $n=120, d=5$ or $n=924,d=29$.
\end{prop}

\begin{proof}
Equation~\eqref{lambda1} together with $|\lambda_1|=d+k$, obtained in the proof of Proposition~\ref{Prop. factorization}, implies
\begin{equation}\label{|d+k|}
d+k=\left|d-1+\mu(1)+\sum_{\substack{m\mid n \\ 1\leq m\leq n-1}}c_m\mu\left(\frac{n}{m}\right)\right|\,,
\end{equation}
where $k\in\mathbb{N}$ is related to $d$ and $n$ by the formula $n=k(2d+k)+1$.
We have $\mu(1)=1$, $\mu(\ell)\in\{1,-1,0\}$ for all $\ell\in\mathbb{N}$ and $\sum_{\substack{m\mid n \\ 1\leq m\leq n}}\left|\mu\left(\frac{n}{m}\right)\right|=2^r$, where $n=q_1^{\alpha_1}q_2^{\alpha_2}\cdots q_r^{\alpha_r}$ is the prime factorization of $n$. Let $s$ denote the number of proper divisors $m$ of $n$ such that $c_m\mu\left(\frac{n}{m}\right)=-1$. Then
$$
\mu(1)+\sum_{\substack{m\mid n \\ 1\leq m\leq n-1}}c_m\mu\left(\frac{n}{m}\right)=\sum_{\substack{m\mid n \\ 1\leq m\leq n}}\left|\mu\left(\frac{n}{m}\right)\right|-2s=2^r-2s\,.
$$
This allows us to rewrite equation~\eqref{|d+k|} in the form
\begin{equation}\label{d k s}
d+k=\left|d-1+2^r-2s\right|\,.
\end{equation}
With \eqref{d k s} in hand, we can proceed to disproving the existence of matrices $C$ for $n=120, d=5$ and $n=924,d=29$.

Let $n=120$, $d=5$. Using equation $n=k(2d+k)+1$ and the prime decomposition of $n=120$, we get $k=7$ and $r=3$ (see the proof of Proposition~\ref{Prop. k small}). Equation~\eqref{d k s} thus takes the form
$$
5+7=\left|5-1+2^3-2s\right|\,.
$$
Hence we have $s=0$ or $s=12$.
Let us start with the case $s=0$. By definition of $s$, equation $s=0$ means that $c_m=\mu\left(\frac{n}{m}\right)$ for every $m<n$ such that $m\mid n$ and $\mu\left(\frac{n}{m}\right)\neq0$. This allows us to find $c_m$ explicitly for each proper divisor $m$ of $n$ that satisfies $\mu\left(\frac{n}{m}\right)=\pm1$.
Knowing $c_m$ for an $m$ being a divisor of $n$, one can use \eqref{gcd} to find values $c_j$  for all $j$ such that $\gcd(j,n)=m$. In this way we obtain Table~\ref{Tab.120,5}. The last column shows all $j\leq\frac{n}{2}$ for which $\gcd(j,n)=m$. Values $c_j$ for $j>\frac{n}{2}$ can be found from the symmetry of $C$ using equation $c_j=c_{n-j}$.
\begin{table}[h]
\begin{center}
\begin{tabular}{|c|c|c|c|l|}
\hline
$m$ & $\frac{n}{m}$ & $\mu\left(\frac{n}{m}\right)$ & $c_m$ & $j\leq\frac{n}{2}\;:\;c_j=c_{n-j}=c_m$ \\
\hline
$1$ & $2^3\cdot3\cdot5$ & $0$ & $c_1$ & $1,7,11,13,17,19,23,29,31,37,41,43,47,49,53,59$ \\
$2$ & $2^2\cdot3\cdot5$ & $0$ & $c_2$ & $2,14,22,26,34,38,46,58$ \\
$3$ & $2^3\cdot5$ & $0$ & $c_3$ & $3,21,33,39,51,57$ \\
$4$ & $2\cdot3\cdot5$ & $-1$ & $-1$ & $4,28,44,52$ \\
$5$ & $2^3\cdot3$ & $0$ & $c_5$ & $5,35,55$ \\
$6$ & $2^2\cdot5$ & $0$ & $c_6$ & $6,42$ \\
$8$ & $3\cdot5$ & $1$ & $1$ & $8,56$ \\
$10$ & $2^2\cdot3$ & $0$ & $c_{10}$ & $10$ \\
$12$ & $2\cdot5$ & $1$ & $1$ & $12$ \\
$15$ & $2^3$ & $0$ & $c_{15}$ & $15$ \\
$20$ & $2\cdot3$ & $1$ & $1$ & $20$ \\
$24$ & $5$ & $-1$ & $-1$ & $24$ \\
$30$ & $2^2$ & $0$ & $c_{30}$ & $30$ \\
$40$ & $3$ & $-1$ & $-1$ & $40$ \\
$60$ & $2$ & $-1$ & $-1$ & $60$ \\
\hline
\end{tabular}
\end{center}
\caption{Values $c_j$ for $n=120$, $d=5$.}
\label{Tab.120,5}
\end{table}
Table~\ref{Tab.120,5} determines the matrix $C$ up to $8$ parameters $c_1,c_2,c_3,c_5,c_6,c_{10},c_{15},c_{30}$ that take values from $\{1,-1\}$. Our computer calculation for each possible $8$-tuple $(c_1,c_2,c_3,c_5,c_6,c_{10},c_{15},c_{30})$ confirmed that the rows of $C$ can never be mutually orthogonal, i.e., a $C$ corresponding to $s=0$ does not exist.

Let us proceed to the case $s=12$. Table~\ref{Tab.120,5} above shows that there are only $7$ proper divisors of $120$ such that $\mu\left(\frac{n}{m}\right)\neq0$, i.e., $s$ cannot exceed $7$. The case $s=12$ is thus impossible. To sum up, there exists no symmetric matrix $C$ satisfying \eqref{Conditions} for $(n,d)=(120,5)$.

Let $n=924$, $d=29$. Then $k=13$ and $r=4$, and equation~\eqref{d k s} takes the form
$$
29+13=\left|29-1+2^4-2s\right|\,;
$$
hence $s=1$ (the other solution, $s=43$, is impossible, because $924$ has only $23$ proper divisors).
Equation $s=1$ means that there is one single proper divisor $m_0$ of $n$ such that
$$
\mu\left(\frac{n}{m_0}\right)\neq0 \quad\text{and}\quad c_{m_0}=-\mu\left(\frac{n}{m_0}\right)\,,
$$
while all the other proper divisors of $n$ satisfy
$$
\left(m_0\neq m<n \quad\text{and}\quad \mu\left(\frac{n}{m}\right)\neq0\right) \quad\text{implies}\quad c_m=\mu\left(\frac{n}{m}\right)\,.
$$
Therefore, for each proper divisor of $n$ such that $\mu\left(\frac{n}{m}\right)\neq0$, we have $c_m=b_m\mu\left(\frac{n}{m}\right)$, where the values $b_m$ form a vector that is a permutation of $(-1,1,1,1,\ldots,1)$. Properties of the M\"obius function $\mu$ imply that the size of the vector is $2^r-1$, i.e., $15$.
Values $c_m$ are shown in Table~\ref{Tab.924,29}. They depend on parameters
$c_1,c_3,c_7,c_{11},c_{21},c_{33},c_{77},c_{231}\in\{1,-1\}$ and on the vector
$$
(b_2,b_4,b_6,b_{12},b_{14},b_{22},b_{28},b_{42},b_{44},b_{66},b_{84},b_{132},b_{154},b_{308},b_{462})\,,
$$
which is a permutation of $(-1,1,1,1,\ldots,1)$.
\begin{table}[h]
\begin{center}
\begin{tabular}{|c|c|c|c||c|c|c|c|}
\hline
$m$ & $\frac{n}{m}$ & $\mu\left(\frac{n}{m}\right)$ & $c_m$ & $m$ & $\frac{n}{m}$ & $\mu\left(\frac{n}{m}\right)$ & $c_m$ \\
\hline
$1$ & $2^2\cdot3\cdot7\cdot11$ & $0$ & $c_1$ & $33$ & $2^2\cdot7$ & $0$ & $c_{33}$ \\
$2$ & $2\cdot3\cdot7\cdot11$ & $1$ & $b_2$ & $42$ & $2\cdot11$ & $1$ & $b_{42}$ \\
$3$ & $2^2\cdot7\cdot11$ & $0$ & $c_3$ & $44$ & $3\cdot7$ & $1$ & $b_{44}$ \\
$4$ & $3\cdot7\cdot11$ & $-1$ & $-b_4$ & $66$ & $2\cdot7$ & $1$ & $b_{66}$ \\
$6$ & $2\cdot7\cdot11$ & $-1$ & $-b_6$ & $77$ & $2^2\cdot3$ & $0$ & $c_{77}$ \\
$7$ & $2^2\cdot3\cdot11$ & $0$ & $c_7$ & $84$ & $11$ & $-1$ & $-b_{84}$ \\
$11$ & $2^2\cdot3\cdot7$ & $0$ & $c_{11}$ & $132$ & $7$ & $-1$ & $-b_{132}$ \\
$12$ & $7\cdot11$ & $1$ & $b_{12}$ & $154$ & $2\cdot3$ & $1$ & $b_{154}$ \\
$14$ & $2\cdot3\cdot11$ & $-1$ & $-b_{14}$ & $231$ & $2^2$ & $0$ & $c_{231}$ \\
$21$ & $2^2\cdot11$ & $0$ & $c_{21}$ & $308$ & $3$ & $-1$ & $-b_{308}$ \\
$22$ & $2\cdot3\cdot7$ & $-1$ & $-b_{22}$ & $462$ & $2$ & $-1$ & $-b_{462}$ \\
$28$ & $3\cdot11$ & $1$ & $b_{28}$ & & & & \\
\hline
\end{tabular}
\end{center}
\caption{Values $c_m$ for $n=924$, $d=29$.}
\label{Tab.924,29}
\end{table}
Entries of $C$ that are not listed in Table~\ref{Tab.924,29} can be obtained using equation~\eqref{gcd}.
A computer calculation shows that for each choice of parameters $c_j$ and $b_j$, the rows of $C$ are not mutually orthogonal. Therefore, a symmetric matrix $C$ of order $924$ satisfying \eqref{Conditions} for $d=29$ does not exist.
\end{proof}

\begin{thm}\label{Thm: symmetric}
If a symmetric matrix $C$ satisfies \eqref{Conditions} for a given $d\geq0$, then $n=2d+2$.
\end{thm}

\begin{proof}
If $d\notin\mathbb{N}_0$ or $d\in\mathbb{N}_0$ is even, then $n=2d+2$ according to results of Section~\ref{Section d<}, see Remark~\ref{Rem.conj}.
If $d=1$, the existence of symmetric circulant Hadamard matrices of orders $n>2d+2=4$ was disproved in papers~\cite{Jo,Br,MKW,CK}.
It remains to verify the statement for odd numbers $d>1$. According to Proposition~\ref{Prop1.d}, the order $n$ obeys $n=k(2d+k)+1$ for some $k\in\mathbb{N}$. However,
\begin{itemize}
\item the case $k>128$ is excluded by Proposition~\ref{Prop. k large};
\item the case $1<k\leq128$ is excluded by Proposition~\ref{Prop. k small}, except for $(k,n,d)=(7,120,5)$ and $(k,n,d)=(13,924,29)$;
\item the existence of symmetric matrices obeying \eqref{Conditions} for $(n,d)=(120,5)$ and $(n,d)=(924,29)$ is disproved by Proposition~\ref{5,29}.
\end{itemize}
To sum up, $k=1$; hence $n=2d+2$.
\end{proof}

We can also formulate a necessary condition for matrices $C$ that are not symmetric; the statement is a direct consequence of Propositions~\ref{Prop1.d} and \ref{Prop.even}:
\begin{prop}
If a matrix $C$ satisfying \eqref{Conditions} is not symmetric, then $d$ is odd and $n\equiv0\pmod4$.
\end{prop}

\section{Matrices $C$ satisfying $n=2d+2$}\label{Section d=}

According to Proposition~\ref{Prop1.d}, the smallest possible order of matrices $C$ obeying conditions~\eqref{Conditions} with a given value $d$ is $n=2d+2$, and
other results of Sections~\ref{Section d<} and \ref{Section: symmetric} indicate that it might be generally the only possible order.
Considering the prominence of matrices $C$ with the property $n=2d+2$, we devote this section to their full characterization. 
Note that the special case when $d$ is not an integer was already solved in Proposition~\ref{Prop1.d} and Remark~\ref{max d exists}.

We will divide the general solution into two steps.
In the first step we examine the situation when $c_{j}=1$ or $c_{n-j}=1$ for all $j=1,\ldots,n-1$ (Proposition~\ref{plus}). In the second step we proceed to the characterization of matrices $C$ such that $c_{m}=c_{n-m}=-1$ for some $m$ (Proposition~\ref{minus}).

\begin{prop}\label{plus}
Let $C$ satisfy \eqref{Conditions} for $n=2d+2$.
Let $c_{j}=1$ or $c_{n-j}=1$ for all $j=1,\ldots,n-1$. Then either
\begin{itemize}
\item $n=2$ and
$$
C=C_2:=\begin{pmatrix}
0 & 1 \\
1 & 0
\end{pmatrix}\,;
$$
\item or $n=4$ and
$$
C=C_{4a}:=\begin{pmatrix}
1 & 1 & 1 & -1 \\
-1 & 1 & 1 & 1 \\
1 & -1 & 1 & 1 \\
1 & 1 & -1 & 1
\end{pmatrix}
\qquad\text{or}\qquad
C=C_{4b}:=\begin{pmatrix}
1 & -1 & 1 & 1 \\
1 & 1 & -1 & 1 \\
1 & 1 & 1 & -1 \\
-1 & 1 & 1 & 1
\end{pmatrix}\,.
$$
\end{itemize}
\end{prop}

\begin{proof}
Proposition~\ref{Prop1.d} implies that $d\in\mathbb{N}_0$ and $n$ is even. Using the assumption $c_{j}=1$ or $c_{n-j}=1$ for the special choice $j=\frac{n}{2}$, we get
\begin{equation}\label{c n/2}
c_{\frac{n}{2}}=1\,.
\end{equation}
Then the orthogonality of the $0$th and the $\frac{n}{2}$th row of $C$ gives the condition
\begin{equation}\label{scalar product}
2\sum_{j=1}^{\frac{n}{2}-1}c_{j}c_{j+\frac{n}{2}}+n-2=0\,.
\end{equation}
Obviously, \eqref{scalar product} is satisfied only if each of the terms $c_jc_{j+\frac{n}{2}}\in\{1,-1\}$ is negative, i.e.,
\begin{equation}\label{antiperiod}
c_{j}=-c_{j+\frac{n}{2}} \quad\text{for all}\; j=1,\ldots,\frac{n}{2}-1\,.
\end{equation}
Now we use equation~\eqref{sum same}, which can be written for $n=2d+2$ in the form
\begin{equation}\label{paired}
d+(c_1+c_{n-1})+(c_2+c_{n-2})+\cdots+(c_{\frac{n}{2}-1}+c_{\frac{n}{2}+1})+1=\pm(d+1)\,.
\end{equation}
With regard to the assumption $c_{j}=1$ or $c_{n-j}=1$ for all $j$, equation~\eqref{paired} can be satisfied only if
\begin{equation}\label{antisym}
c_{j}=-c_{n-j} \quad\text{for all}\; j=1,\ldots,\frac{n}{2}-1\,.
\end{equation}
Equation~\eqref{c n/2} and conditions~\eqref{antiperiod} and~\eqref{antisym} imply that $C$ has the block form
\begin{equation}\label{C block A}
C=\left(\begin{array}{cc}
dI+A & I-A \\
I-A & dI+A
\end{array}\right)\,,
\end{equation}
where $A$ is a Toeplitz matrix with the $0$th row equal to $(0,c_{1},\ldots,c_{\frac{n}{2}-1})$ and with the $0$th column equal to $(0,-c_{1},\ldots,-c_{\frac{n}{2}-1})^T$. Therefore, $A=-A^T$. %is antisymmetric.
Equation~\eqref{C block A} together with the antisymmetry of $A$ implies
\begin{equation}\label{CCt}
CC^T=\left(\begin{array}{cc}
(d^2+1)I+2AA^T & 2dI-2AA^T \\
2dI-2AA^T & (d^2+1)I+2AA^T
\end{array}\right)\,.
\end{equation}
With regard to \eqref{CCt}, the condition $CC^T=(d^2+n-1)I$ is equivalent to
$AA^T=dI$.
Combining this fact with $A=-A^T$, we get
\begin{equation*}
(A-I)(A-I)^T=(d+1)I=\frac{n}{2}I\,,
\end{equation*}
i.e., $A-I$ is an Hadamard matrix.
Hence we get three possibilities:
\begin{itemize}
\item $\frac{n}{2}=1$ and $A-I=(-1)$. Substituting $A=(0)$ into \eqref{C block A}, we obtain the solution $C_2$.
\item $\frac{n}{2}=2$ and $A-I$ is either $\left(\begin{smallmatrix}-1&1\\-1&-1\end{smallmatrix}\right)$ or $\left(\begin{smallmatrix}-1&1\\-1&-1\end{smallmatrix}\right)$. When we substitute
$A=\left(\begin{smallmatrix}0&1\\-1&0\end{smallmatrix}\right)$ and $A=\left(\begin{smallmatrix}0&1\\-1&0\end{smallmatrix}\right)$
into~\eqref{C block A}, we obtain the solutions $C_{4a}$ and $C_{4b}$, respectively.
\item $\frac{n}{2}\geq4$ is a multiple of $4$.
\end{itemize}
In order to show that there is no solution for $\frac{n}{2}\geq4$,
let us use \eqref{antiperiod} and \eqref{antisym} to derive the relation
\begin{equation}\label{A satisfies lemma}
c_{j}=-c_{j+\frac{n}{2}}=-(-c_{n-(j+\frac{n}{2})})=c_{\frac{n}{2}-j} \qquad\text{for all}\; j=1,\ldots,\frac{n}{2}-1.
\end{equation}
Relation~\eqref{A satisfies lemma} implies that the $0$th row of $A-I$ takes the form
$$
\begin{array}{cccccccccccccl}
-1 & c_1 & c_2 & c_3 & \cdots & c_{\frac{n}{4}-1} & c_{\frac{n}{4}} & c_{\frac{n}{4}-1} & c_{\frac{n}{4}-2} & c_{\frac{n}{4}-3} & \cdots & c_2 & c_1 & , \\
\end{array}
$$
thus the $1$st and $2$nd row of $A-I$ read
$$
\begin{array}{cccccccccccccl}
-c_1 & -1 & c_1 & c_2 & \cdots & c_{\frac{n}{4}-2} & c_{\frac{n}{4}-1} & c_{\frac{n}{4}} & c_{\frac{n}{4}-1} & c_{\frac{n}{4}-2} & \cdots & c_3 & c_2 \\
-c_2 & -c_1 & -1 & c_1 & \cdots & c_{\frac{n}{4}-3} & c_{\frac{n}{4}-2} & c_{\frac{n}{4}-1} & c_{\frac{n}{4}} & c_{\frac{n}{4}-1} & \cdots & c_4 & c_3 & .
\end{array}
$$
The scalar product of the $0$th with the $1$st row is equal to
$$
2\left(\sum_{j=1}^{\frac{n}{4}-1}c_{j}c_{j+1}\right)\,.
$$
Similarly, the scalar product of the $0$th with the $2$nd row equals
$$
-c_1^2+c_{\frac{n}{4}-1}^2+2\left(\sum_{j=1}^{\frac{n}{4}-2}c_{j}c_{j+2}\right)=2\left(\sum_{j=1}^{\frac{n}{4}-2}c_{j}c_{j+2}\right)\,.
$$
Both scalar products should be zero. Hence we obtain the requirement
$$
\sum_{j=1}^{\frac{n}{4}-1}c_{j}c_{j+1}=0 \qquad\text{and}\qquad \sum_{j=1}^{\frac{n}{4}-2}c_{j}c_{j+2}=0\,.
$$
However, since the two sums have different parities, they cannot vanish at the same time.
\end{proof}

\begin{prop}\label{minus}
Let $C$ satisfy \eqref{Conditions} for $n=2d+2$.
Let there be an $m\in\{1,\ldots,n-1\}$ such that $c_{m}=c_{n-m}=-1$. Then $C$ is a block circulant matrix taking the form
\begin{equation}\label{C block}
C=\left(\begin{array}{ccccc}
B+\frac{n}{2}I & B & B & \cdots & B \\
B & B+\frac{n}{2}I & B & \cdots & B \\
B & B & B+\frac{n}{2}I &  & B \\
\vdots & \vdots &  & \ddots &  \\
B & B & B &  & B+\frac{n}{2}I
\end{array}\right)\,,
\end{equation}
where the block $B$ is either the $1\times1$ matrix $(-1)$ or $B$ is one of the matrices
\begin{equation}\label{B}
C_2-I\,,\quad C_{4a}-2I\,,\quad C_{4b}-2I
\end{equation}
for $C_2$, $C_{4a}$, $C_{4b}$ defined in Proposition~\ref{plus}.
\end{prop}

\begin{proof}
Let $m$ be the minimal number with the property $c_{m}=c_{n-m}=-1$.
The $0$th and the $m$th row of $C$ take the form
$$
\begin{array}{ccccccccccl}
\left(\frac{n}{2}-1\right) & c_{1} & c_{2} & \cdots & c_{m-1} & -1 & c_{m+1} & \cdots & c_{n-2} & c_{n-1} \\
-1 & c_{n-m+1} & c_{n-m+2} & \cdots & c_{n-1} & \left(\frac{n}{2}-1\right) & c_{1} & \cdots & c_{n-m-2} & c_{n-m-1} & .
\end{array}
$$
Their scalar product must be zero; hence
\begin{equation}\label{scalarproduct}
\sum_{j=n-m+1}^{n-1}c_{j}c_{j+m-n}+\sum_{j=1}^{n-m-1}c_{j}c_{j+m}=n-2\,.
\end{equation}
Equation~\eqref{scalarproduct} is satisfied if and only if all the summands on the left hand side are equal to $1$, i.e.,
\begin{equation}\label{period}
c_{j}=c_{(j+m) \bmod n} \quad\text{for all}\; j=1,\ldots,n-1,\; j\neq n-m\,.
\end{equation}
Equation~\eqref{period} implies that $m$ divides $n$. Indeed, if $(n \bmod m)=k\neq0$, we would get
\begin{equation*}
-1=c_m=c_{2m}=\cdots=c_{n-k} \qquad\text{and}\qquad -1=c_{n-m}=c_{n-2m}=\cdots=c_{k}\,,
\end{equation*}
i.e., $c_k=c_{n-k}=-1$ for some $k<m$, which would contradict the definition of $m$.
By equation~\eqref{period}, the generator of $C$ takes the form
\begin{equation*}
\left(\frac{n}{2}-1,c_1,\ldots,c_{m-1},-1,c_1,\ldots,c_{m-1},-1,c_1,\ldots,c_{m-1},\ldots,-1,c_1,\ldots,c_{m-1}\right)\,.
\end{equation*}
Consequently, $C$ has the block form~\eqref{C block} for $B$ being a circulant matrix with generator $(-1,c_1,\ldots,c_{m-1})$. If $m=1$, we obtain immediately $B=(-1)$. If $m\geq2$, the minimality of $m$ trivially implies that
\begin{equation}\label{B c}
c_j=1\;\;\text{or}\;\; c_{m-j}=1 \quad\text{for all $j=1,\ldots,m-1$}.
\end{equation}
By assumption, matrix $C$ satisfies $CC^T=(d^2+n-1)I$ with $d=\frac{n}{2}-1$, i.e., $CC^T=\frac{n^2}{4}I$.
Hence we get the condition
$$
\frac{n}{m}BB^T+\frac{n}{2}(B+B^T)=0\,,
$$
which is equivalent to
\begin{equation}\label{B orthogonal}
\left(B+\frac{m}{2}I\right)\left(B+\frac{m}{2}I\right)^T=\frac{m^2}{4}I\,.
\end{equation}
To sum up, if $m\geq2$, then $B+\frac{m}{2}I$ is an $m\times m$ circulant matrix with generator $(-1+\frac{m}{2},c_1,\ldots,c_{m-1})$ and with properties~\eqref{B c} and \eqref{B orthogonal}. Since the matrix $B+\frac{m}{2}I$ satisfies all assumptions of Proposition~\ref{plus}, $B+\frac{m}{2}I$ equals $C_2$, $C_{4a}$ or $C_{4b}$. Hence we obtain the three possibilities listed in \eqref{B}.
\end{proof}

According to Propositions \ref{plus} and \ref{minus},
a matrix $C$ satisfies conditions~\eqref{Conditions} with $n=2d+2$ if and only if the generator of $C$ takes one of the forms below.
\begin{align*}
g_1=&\left(\frac{n}{2}-1,-1,-1,\ldots,-1\right) \quad\text{(for any $n\geq2$)}\,; \\
g_2=&\left(\frac{n}{2}-1,1,-1,1,-1,1,\ldots,-1,1\right) \quad\text{(for even $n$)}\,; \\
g_{4a}=&\left(\frac{n}{2}-1,1,1,-1,-1,1,1,-1,\ldots,-1,1,1,-1\right) \quad\text{(for $n$ being a multiple of $4$)}\,; \\
g_{4b}=&\left(\frac{n}{2}-1,-1,1,1,-1,-1,1,1,\ldots,-1,-1,1,1\right) \quad\text{(for $n$ being a multiple of $4$)}\,.
\end{align*}
In particular, a matrix $C$ with $n=2d+2$ exists for every $n\geq2$.
Note that $C$ may or may not be symmetric:
\begin{itemize}
\item If $C$ has generator $g_1$ or $g_2$, then $C^T=C$.
\item If $C_a$, $C_b$ are circulant matrices of the same order with generators $g_{4a}$ and $g_{4b}$, respectively, then $C_a^T=C_b\neq C_a$.
\end{itemize}

\begin{rem}
The case $d=1$ yields the only known circulant Hadamard matrices---of order $4$---in keeping with the circulant Hadamard matrix conjecture.
\end{rem}

If Conjecture~\ref{Conjecture} is true, then generators $g_1,g_2,g_{4a},g_{4b}$ listed above determine all the matrices $C$ satisfying \eqref{Conditions}, giving thus a complete solution to the problem.

\section*{Acknowledgements}
We thank J.~Seberry and R.~Craigen for useful comments on the topic and to the referee for many suggestions that helped us to improve the paper.
O.~T.\ appreciates the hospitality at Jagiellonian University in Krakow, where a part of this work was done,
and a support from the Czech Science Foundation (GA\v{C}R) within the project 17-01706S.
D.~G.\ acknowledges Grant FONDECYT Iniciaci\'{o}n number 11180474, Chile.


\begin{thebibliography}{00}

\bibitem{AAMS}
M.~H.~Ang, K.~T.~Arasu, S.~L.~Ma, Y.~Strassler, Study of proper circulant weighing matrices with weight $9$, \textit{Discrete Math.} \textbf{308} (2008) 2802--2809.

\bibitem{ALMNR}
K.~T.~Arasu, K.~H.~Leung, S.~L.~Ma, A.~Nabavi, D.~K.~Ray-Chaudhuri, Determination of all possible orders of weight $16$ circulant weighing matrices, \textit{Finite Fields Th. App.} \textbf{12} (2006) 498--538.

\bibitem{AS}
K.~T.~Arasu, J.~Seberry, On circulant weighing matrices, \textit{Australasian J. Combin.} \textbf{17} (1998) 21--37.

\bibitem{BJL99} T. Beth, D. Jungnickel, H. Lenz, \emph{Design Theory} (2nd edition), Cambridge University Press, 1999.

\bibitem{BM08} P. Borwein, M. Mossinghoff, Barker sequences and flat polynomials, in: J. McKee, C. Smyth (Eds.), \emph{Number Theory and Polynomials} (Bristol, U.K., 2006). London Math. Soc. Lecture Note Ser., vol.~352, Cambridge Univ. Press, 2008, pp. 71--88.

\bibitem{Br}
R.~A.~Brualdi, A note on multipliers of difference sets, \textit{J. Res. Natl. Bur. Stand.} \textbf{69B} (1965) 87--89.

\bibitem{CB67} C. Cook, M. Bernfeld, \emph{Radar signals: An Introduction to Theory and Application}, Academic Press, New York, 1967.

\bibitem{Cr}
R.~Craigen, Trace, symmetry and orthogonality, \textit{Canad. Math. Bull.} \textbf{37} (1994) 461--467.

\bibitem{CK}
R.~Craigen, H.~Kharaghani, On the nonexistence of Hermitian circulant complex Hadamard matrices, \textit{Australas. J. Combin.} \textbf{7} (1993) 225--227.

\bibitem{EH}
P.~Eades, R.M.~Hain, On circulant weighing matrices, \textit{Ars Combin.} \textbf{2} (1976) 265--284.

\bibitem{GT}
D.~Goyeneche, O.~Turek, Equiangular tight frames and unistochastic matrices, \emph{J.\ Phys.\ A: Math.\ Theor.} \textbf{50} (2017) 245304.

\bibitem{HW}
G.~H.~Hardy, E.~M.~Wright, \textit{An Introduction to the Theory of Numbers} (5th ed.), Oxford University Press, Oxford, 1980.

\bibitem{H52} D. Huffman, A method for the construction of minimum redundancy codes, \emph{Proc. of the IRE}, vol.~40 (1952) pp. 1098--1101.

\bibitem{H62} D. Huffman, The generation of impulse-equivalent pulse trains, \emph{IRE Trans. on Information Theory}, vol.~8 (1962) 10--16.

\bibitem{Jo}
E.~C.~Johnsen, The inverse multiplier for abelian group difference sets, \textit{Canad. J. Math.} \textbf{16} (1964) 787--796.

\bibitem{KO}
P.~Kurasov, R.~Ogik, On equi-transmitting matrices, \textit{Research Reports in Mathematics}, no. 1 (2014), Stockholm University.

\bibitem{La}
C.~W.~H.~Lam, Non-skew symmetric orthogonal matrices with constant diagonals, \textit{Discrete Math.} \textbf{43} (1983) 65--78.

\bibitem{MKW}
J.~H.~McKay, S.S.-S.~Wang, On a theorem of Brualdi and Newman, \textit{Linear Algebra Appl.} \textbf{92} (1987) 39--43.

\bibitem{Ry}
H.~J.~Ryser, \textit{Combinatorial mathematics}, Willey, New York, 1963.

\bibitem{Sch1}
B.~Schmidt, Cyclotomic integers and finite geometry, \textit{J. Am. Math. Soc.} \textbf{12} (1999) 929--952.

\bibitem{Sch2}
B.~Schmidt, Towards Ryser's conjecture, in: C. Casacuberta et al., eds., \textit{Proc. of 3rd European Congress on Mathematics}, Progress in Mathematics, vol.~201, Birkh\"auser 2001, pp. 533--541.

\bibitem{SL}
J.~Seberry, C.~W.~H.~Lam, On orthogonal matrices with constant diagonal, \textit{Linear Algebra Appl.} \textbf{46} (1982) 117--129.

\bibitem{SM}
R.~G.~Stanton, R.~C.~Mullin, On the nonexistence of a class of circulant balanced weighing matrices, \textit{SIAM J. Appl. Math.} \textbf{30} (1976) 98--102.

\bibitem{TC}
O.~Turek, T.~Cheon, Hermitian unitary matrices with modular permutation symmetry, \textit{Linear Algebra Appl.} \textbf{469} (2015) 569--593.

\bibitem{Tu}
R.~Turyn, Character sums and difference sets, \textit{Pacific J. Math.} \textbf{15} (1965) 319--346.

\end{thebibliography}
\end{document}